\documentclass[12pt]{article}
\usepackage[T1]{fontenc}
\usepackage[cp1250]{inputenc}
\usepackage{lmodern}

\usepackage[english]{babel}
\usepackage{latexsym}
\usepackage{amsmath}
\usepackage{indentfirst}
\usepackage{amsthm}
\usepackage{amssymb}
\usepackage{amsfonts}
\usepackage{stackrel}
\usepackage{dsfont}
\usepackage{tikz}
\usepackage{slashbox}
\usepackage[mathscr]{eucal}
\usepackage{authblk}

\newtheorem{thm}{Theorem}
\newtheorem{lemma}[thm]{Lemma}

\newcommand{\reals}{\mathbb{R}}

\newcommand{\complex}{\mathbb{C}}

\newcommand{\eps}{\varepsilon}
\newcommand{\supp}{\text{supp}}

\newcommand{\Real}{\text{Re}}
\newcommand{\Imag}{\text{Im}}

\newcommand{\Ffamily}{\mathcal{F}}
\newcommand{\Afamily}{\mathcal{A}}
\newcommand{\Laplace}{\mathcal{L}}

\begin{document}
\title{Characterizing compact families via the Laplace transform}
\author{Mateusz Krukowski}
\affil{\L\'od\'z University of Technology, Institute of Mathematics, \\ W\'ol\-cza\'n\-ska 215, \
90-924 \ \L\'od\'z, \ Poland}
\maketitle

\begin{abstract}
In 1985, Robert L. Pego characterized compact families in $L^2(\reals)$ in terms of the Fourier transform. It took nearly 30 years to realize that Pego's result can be proved in a wider setting of locally compact abelian groups (works of G\'orka and Kostrzewa). In the current paper, we argue that the Fourier transform is not the only integral transform that is efficient in characterizing compact families and suggest the Laplace transform as a possible alternative.   
\end{abstract}

\vspace{0.2cm}

\section{Introduction}

Characterizing compact families is a vital topic in function spaces' theory at least since the end of the 19-th century. Around 1883, two Italian mathematicians Cesare Arzel\`a (1847-1912) and Giulio Ascoli (1843-1896) provided the necessary and sufficient conditions under which every sequence of a given family of real-valued continuous functions (defined on a closed and bounded interval), has a uniformly convergent subsequence (this is called \textit{sequential compactness}). A couple of decades later (in 1931), Andrey Kolmogorov (1903-1987) succeeded in characterizing the compact families in $L^p(\reals^N),$ when $1 < p < \infty$ and all the functions are supported in a common bounded set (comp. \cite{Kolmogorov}). A year later, Jacob David Tamarkin (1888-1945) got rid of the second restriction (comp. \cite{Tamarkin}) and in 1933, Marcel Riesz (1886-1969), a younger brother of Frigyes Riesz, proved the general case for $L^p(\reals^N),$ where $1\leqslant p < \infty.$ In 1940, a French mathematician and one of the leaders of the Bourbaki group, Andr\'e Weil (1906-1998) wrote a book '\textit{L'int\'{e}gration dans les groupes topologique}' (comp. \cite{Weil}), in which he proved the Kolmogorov-Riesz theorem for a locally compact Hausdorff group $G$ instead of $\reals^N.$

The next major contribution came over 40 years later (1985), when Robert L. Pego characterized compact families in $L^2(\reals)$ via the Fourier transform. This innovative idea was the cornerstone for the works of two Polish mathematicians Przemys\l{}aw G\'orka and Tomasz Kostrzewa. In \cite{Gorka} and \cite{GorkaKostrzewa}, they proved that a counterpart of Pego theorem holds for locally compact abelian groups (this is reminiscent of Weil's contribution). Obviously, there are other works related to the topic, which are worth-mentioning: \cite{GorkaMacios}, \cite{GorkaMacios2}, \cite{GorkaRafeiro}, \cite{GorkaRafeiro2}, \cite{HancheOlsenHolden} or \cite{HancheOlsenHoldenMalinnikova} just to name a few. 

In the current paper, we argue that the Fourier transform is not the only one that can be used to characterize compact families. In Section \ref{prelims} we introduce the basic definitions and discuss the necessary notation. We also prove the fundamental theorems, which are very well-known in the context of the Fourier transform, and probably less known in the context of the Laplace transform. In Section \ref{mainresults} we prove the main results. Theorem \ref{LaplacePego}, which is a counterpart of the Pego's result, is the climax of the paper.

\section{Preliminary results}
\label{prelims}

Throughout the paper, by $\reals_+$ we understand the open set $(0,\infty)$ and $\complex$ stands for the field of complex numbers. For a measurable, complex-valued function $f:\reals_+\longrightarrow \complex$ and a real number $x\geqslant 0$ we denote 
$$f_x(t) = f(t)e^{-xt}.$$

\noindent
We say that $f:\reals_+\longrightarrow\complex$ is a \textit{Laplace-Pego function} of order $x\geqslant 0$ if
$$f_x \in L^1(\reals_+)\cap L^2(\reals_+).$$

\noindent
The norms in $L^1(\reals_+)$ and $L^2(\reals_+)$ are denoted by $\|\cdot\|_1$ and $\|\cdot\|_2$, respectively. 
Moreover, if $\Ffamily$ is a family of Laplace-Pego functions with a common order $x\geqslant 0$, then we denote
$$\Ffamily_x = \{f_x \ : \ f\in\Ffamily\}.$$

Let $f$ be a Laplace-Pego function of order $x\geqslant 0$. The \textit{Laplace transform} $\Laplace\{f\}$ of the function $f$ is defined by
$$\Laplace\{f\}(z) = \int_0^{\infty}\ f(t)e^{-zt}\ dt.$$

\noindent
A natural question arises: \textit{when does the above integral exist}? To answer this question, observe that if $\Real(z) \geqslant x$, then
\begin{gather*}
|\Laplace\{f\}(z)| = \left|\int_0^{\infty}\ f(t)e^{-\Real(z)t}e^{-i\Imag(z)t}\ dt\right| \leqslant \int_0^{\infty}\ |f(t)|e^{-\Real(z)t}\ dt \\
= \int_0^{\infty}\ |f(t)|e^{-xt}e^{(x-\Real(z))t}\ dt \leqslant \|f_{x}\|_1 < \infty.
\end{gather*}

\noindent
In other words, the Laplace transform $\Laplace\{f\}$ exists in the half-plane $\Real(z)\geqslant x$.

An important special case of the Laplace transform is the \textit{Fourier transform}, which we define by
$$\widehat{f}(y) = \Laplace\{f\}(2\pi iy).$$

Let us formulate a crucial theorem regarding the Laplace transform, which we will use multiple times throughout the paper:

\begin{thm}(Plancherel theorem for the Laplace transform)\\
If $f$ is a Laplace-Pego function of order $x\geqslant 0$, then
\begin{gather}
\frac{1}{2\pi}\ \int_{-\infty}^{\infty}\ |\Laplace\{f\}(x+iy)|^2\ dy = \int_0^{\infty}\ e^{-2xt} |f(t)|^2\ dt.
\label{Plancherel}
\end{gather}
\label{Planchereltheorem}
\end{thm}
\begin{proof}
At first, observe that 
\begin{equation}
\begin{split}
\forall_{y\in\reals}\ \Laplace\{f\}(x+iy) &= \int_0^{\infty}\ f(t)e^{-xt}e^{-iyt}\ dt \\
&= \int_0^{\infty}\ f(t)e^{-xt}e^{-2\pi i\frac{y}{2\pi}t}\ dt= \widehat{f_x}\left(\frac{y}{2\pi}\right) .
\end{split}
\label{Plancherelequality}
\end{equation}

\noindent
By the classical Plancherel theorem (Theorem 3.5.2 in \cite{Deitmar}, p. 53 or \mbox{Theorem 1.1} in \cite{SteinShakarchi}, p. 208) we have
\begin{gather*}
\int_{-\infty}^{\infty}\ \left|\widehat{f_x}\left(\frac{y}{2\pi}\right)\right|^2\ d\frac{y}{2\pi} = \int_{-\infty}^{\infty}\ |f_x(t)|^2\ dt = \int_0^{\infty}\ e^{-2xt}|f(t)|^2\ dt.
\end{gather*}

\noindent
Upon observing that 
$$\int_{-\infty}^{\infty}\ \left|\widehat{f_x}\left(\frac{y}{2\pi}\right)\right|^2\ d\frac{y}{2\pi} \stackrel{(\ref{Plancherelequality})}{=} \frac{1}{2\pi}\ \int_{-\infty}^{\infty}\ |\Laplace\{f\}(x+iy)|^2\ dy$$

\noindent
we conclude the proof. 
\end{proof}

The theorem, which we present below, is (again) a counterpart of a well-know result in the theory of Fourier transform:

\begin{thm}(Riemann-Lebesgue lemma for the Laplace transform)\\
If $f$ is a Laplace-Pego function of order $x\geqslant 0$, then
\begin{gather}
\lim_{y\rightarrow \pm\infty}\ \Laplace\{f\}(x+iy) = 0.
\label{RiemannLebesgue}
\end{gather}
\label{RiemannLebesguelemma}
\end{thm}
\begin{proof}
At first, let $f = \mathds{1}_{(a,b)}$ where $(a,b)\subset \reals_+.$ Then
$$\forall_{y\in\reals}\ \Laplace\{f\}(x+iy) = \int_0^{\infty}\ \mathds{1}_{(a,b)}(t)e^{-(x+iy)t}\ dt = \int_a^b\ e^{-(x+iy)t}\ dt = \frac{e^{-(x+iy)a} - e^{-(x+iy)b}}{x+iy},$$

\noindent
so (\ref{RiemannLebesgue}) holds. By linearity of the Laplace transform, the result is also true for all simple functions.

Finally, let $f$ be an arbitrary Laplace-Pego function and let $\eps> 0$. Since simple functions are dense in $L^1(\reals_+)$ (comp. Proposition 6.7 in \cite{FollandRealAnalysis}, p. 183) there exists a simple function $g$ such that
\begin{gather}
\int_0^{\infty}\ |f(t)e^{-xt} - g(t)|\ dt < \eps.
\label{simplefunctiongdef}
\end{gather}

\noindent
Hence
\begin{gather*}
\lim_{y\rightarrow \pm \infty}\ |\Laplace\{f\}(x+iy)| = \lim_{y\rightarrow \pm \infty}\ \left|\int_0^{\infty}\ f(t)e^{-xt}e^{-iyt}\ dt\right| \\
\leqslant \int_0^{\infty}\ \left|f(t)e^{-xt} - g(t)\right|\ dt + \lim_{y\rightarrow \pm \infty}\ \left|\int_0^{\infty}\ g(t)e^{-iyt}\ dt\right| \\
\stackrel{(\ref{simplefunctiongdef})}{<} \eps  + \lim_{y\rightarrow \pm \infty}\ |\Laplace\{g\}(iy)| = \eps + \lim_{y\rightarrow\pm \infty}\ \left|\widehat{g}\left(\frac{y}{2\pi}\right)\right| = \eps,
\end{gather*}

\noindent
where the last equality follows from the classical Riemann-Lebesgue lemma for the Fourier transform (comp. Theorem 1.7 in \cite{Katznelson}, p. 136). Since $\eps >0$ was chosen arbitrarily, we conclude the proof.
\end{proof}

We will now recall the prominent fact that the Laplace trasnform 'changes the convolution of two functions to multiplication'. A convolution of two Laplace-Pego functions $f,\ g$ with a common order $x\geqslant 0$ is defined by 
$$\forall_{t > 0}\ f\star g(t) = \int_0^t\ f(s)g(t-s)\ ds.$$

\noindent
In order to prove that the convolution is well-defined, let us invoke a general version of Tonelli's theorem (comp. \mbox{Theorem B.3.3} in \cite{DeitmarEchterhoff}, p. 289):

\begin{thm}(Tonelli's theorem)\\
Let $(X,\mu_X),\ (Y,\mu_Y)$ be two measure spaces and let $F : X\times Y\longrightarrow \complex$ be a measurable function such that 
$$\big\{(x,y)\in X\times Y\ : \ F(x,y)\neq 0\big\}$$ 

\noindent
is $\sigma-$finite. If one of the integrals 
$$\int_X\ \int_Y\ |F(x,y)|\ d\mu_X(x)\ d\mu_Y(y) \hspace{0.4cm}\text{or}\hspace{0.4cm} \int_X\ \int_Y\ |F(x,y)|\ d\mu_X(y)\ d\mu_Y(x)$$

\noindent
is finite, then
$$\int_X\ \int_Y\ F(x,y)\ d\mu_X(x)\ d\mu_Y(y) = \int_X\ \int_Y\ F(x,y)\ d\mu_Y(y)\ d\mu_X(x).$$
\end{thm}

In our case, both $X$ and $Y$ are the space $\reals_+$ and both measures $\mu_X$ and $\mu_Y$ are the standard Lebesgue measure. Consequently, the assumption of 
$$\big\{(x,y)\in\reals_+\times\reals_+\ :\ F(x,y)\neq 0\big\}$$ 

\noindent
being $\sigma-$finite becomes obsolete, since $\reals_+\times\reals_+$ is $\sigma-$finite.

\begin{thm}
If $f,\ g$ are Laplace-Pego functions with a common order $x\geqslant 0,$ then $(f\star g)_x \in L^1(\reals_+)$. In particular, it exists almost everywhere.
\end{thm}
\begin{proof}
At first, let us observe that
\begin{equation}
\begin{split}
\forall_{t > 0}\ e^{-xt} f\star g(t) &= \int_0^t\ e^{-xt}f(s)g(t-s)\ ds \\
& = \int_0^t\ e^{-xs}f(s)e^{-x(t-s)}g(t-s)\ ds.
\end{split}
\label{convoexp}
\end{equation}

\noindent
Furthermore, by Proposition 3.9, p. 86 in \cite{SteinShakarchi}, we note that the function \mbox{$F:\reals_+\times\reals_+\longrightarrow \complex$} defined by
$$F(t,s) = e^{-xs}f(s)e^{-x(t-s)}g(t-s)$$

\noindent
is measurable, so we are in position to apply Tonelli's theorem. Consequently, we obtain
\begin{equation*}
\begin{split}
\bigg|\int_0^{\infty}\ e^{-xt}f\star g(t)\ dt\bigg| \stackrel{(\ref{convoexp})}{\leqslant}& \int_0^{\infty}\ \int_0^t\ e^{-xs}|f|(s)e^{-x(t-s)}|g|(t-s)\ ds\ dt \\
\stackrel{\text{Tonelli's thm}}{=}& \int_0^{\infty}\ \int_s^{\infty}\ e^{-xs}|f|(s)e^{-x(t-s)}|g|(t-s)\ dt\ ds = \|f_x\|_1\ \|g_x\|_1 < \infty,
\end{split}
\end{equation*}

\noindent
which ends the proof.
\end{proof}

\begin{thm}(convolution theorem for the Laplace transform, comp. Theorem 2.39 in \cite{Schiff}, p. 92)\\
If $f$ and $g$ are Laplace-Pego functions with a common order $x\geqslant 0,$ then 
$$\Laplace\{f\star g\}(z) = \Laplace\{f\}(z)\cdot \Laplace\{g\}(z)$$

\noindent
for $\Real(z)\geqslant x$.
\label{convolutiontheorem}
\end{thm}
\begin{proof}
By Proposition 3.9, p. 86 in \cite{SteinShakarchi}, we note that the function \mbox{$F:\reals_+\times\reals_+\longrightarrow\complex$} defined by
$$F(t,s) = f(s)g(t-s)e^{-zt}$$

\noindent
is measurable, so we are in position to apply Tonelli's theorem. We have
\begin{equation*}
\begin{split}
\Laplace\{f\star g\}(z) &= \int_0^{\infty}\ f\star g(t)e^{-zt}\ dt = \int_0^{\infty}\ \int_0^t\ f(s)g(t-s)e^{-zt}\ ds\ dt \\
&\stackrel{\text{Tonelli's thm}}{=} \int_0^{\infty}\ \int_s^{\infty}\ f(s)g(t-s)e^{-zt}\ dt\ ds = \Laplace\{f\}(z)\cdot \Laplace\{g\}(z),
\end{split}
\end{equation*}

\noindent 
which ends the proof. 
\end{proof}

\section{Main results}
\label{mainresults}

A family $\Ffamily$ of Laplace-Pego functions with a common order $x\geqslant 0$ is said to be \textit{exponentially $L^2-$equivanishing} at $x$, if the family $\Ffamily_x$ is $L^2-$equivanishing, i.e.
\begin{gather}
\forall_{\eps >0}\ \exists_{T>0}\ \forall_{f\in\Ffamily}\ \int_T^{\infty}\ e^{-2xt}|f(t)|^2\ dt < \eps.
\label{expoL2equivanishing}
\end{gather}

\noindent
Furthermore, we say that a family $\Ffamily$ is \textit{Laplace equicontinuous} at $x$, if
\begin{gather}
\forall_{\eps > 0}\ \exists_{\delta >0}\ \forall_{f\in\Ffamily}\ \frac{1}{2\pi}\ \int_{-\infty}^{\infty}\ |\Laplace\{f\}(x+iy+\delta) - \Laplace\{f\}(x+iy)|^2\ dy < \eps.
\label{Laplaceequicontinuity}
\end{gather}

We will now relate the concepts of Laplace equicontinuity and exponential $L^2-$equivanishing. 

\begin{thm}
Let $\Ffamily$ be a Laplace-Pego family with a common order $x\geqslant 0.$ If $\Ffamily$ is Laplace equicontinuous at $x$, then it is exponentially $L^2-$equivanishing at $x$. Furthermore, if $\Ffamily_x$ is $L^2-$bounded, then the implication can be reversed. 
\label{LaplaceequicontinuityexpoL2equvanishing}
\end{thm}
\begin{proof}
We divide the proof into two parts:

\vspace{0.3cm}
\noindent
\textbf{Part 1.}

At first, we assume that $\Ffamily$ is Laplace equicontinuous at $x$, so for a fixed $\eps > 0$ we may choose $\delta > 0$ according to (\ref{Laplaceequicontinuity}). Let $T>0$ be such that 
\begin{gather}
\left|e^{-\delta T} - 1\right|^2 \geqslant \frac{1}{2}.
\label{choiceofT}
\end{gather}

\noindent
Consequently, for every $f\in\Ffamily$ we obtain
\begin{gather*}
\eps > \frac{1}{2\pi}\ \int_{-\infty}^{\infty}\ |\Laplace\{f\}(x+iy+\delta) - \Laplace\{f\}(x+iy)|^2\ dy \\
= \frac{1}{2\pi}\ \int_{-\infty}^{\infty}\ \left| \int_0^{\infty}\ f(t)\left(e^{-\delta t} - 1\right)e^{-(x+iy)t}\ dt \right|^2\ dy\\
\stackrel{\text{Theorem } \ref{Planchereltheorem}}{=} \int_0^T\ e^{-2xt}|f(t)|^2 \left|e^{-\delta t} - 1\right|^2\ dt + \int_T^{\infty}\ e^{-2xt}|f(t)|^2 \left|e^{-\delta t} - 1\right|^2\ dt\\
\stackrel{(\ref{choiceofT})}{\geqslant} \frac{1}{2} \int_T^{\infty}\ e^{-2xt} |f(t)|^2\ dt,
\end{gather*} 

\noindent
which ends the first part of the proof.

\vspace{0.3cm}
\noindent
\textbf{Part 2.}

At this point, we assume that $\Ffamily_x$ is $L^2-$bounded, so there exists $M>0$ such that 
\begin{gather*}
\forall_{f\in\Ffamily}\ \int_0^{\infty}\ e^{-2xt} |f(t)|^2\ dt < M.
\end{gather*}

\noindent
We will show that if $\Ffamily$ is exponentially $L^2-$equivanishing at $x$, then it is Laplace equicontinuous at $x$.

Fix $\eps>0$ and choose $T>0$ as in the definition of the exponential $L^2-$equivanishing (\ref{expoL2equivanishing}). Let $\delta>0$ be such that 
\begin{gather}
\forall_{f\in\Ffamily}\ \left|e^{-\delta T} - 1\right|^2 M < \eps. 
\label{choiceofdelta}
\end{gather}

\noindent
We have
\begin{gather*}
\frac{1}{2\pi}\ \int_{-\infty}^{\infty}\ |\Laplace\{f\}(x+iy+\delta) - \Laplace\{f\}(x+iy)|^2\ dy \\
\stackrel{\text{Theorem } \ref{Planchereltheorem}}{=} \int_0^T\ e^{-2xt}|f(t)|^2 \left|e^{-\delta t} - 1\right|^2\ dt + \int_T^{\infty}\ e^{-2xt}|f(t)|^2 \left|e^{-\delta t} - 1\right|^2\ dt\\
\stackrel{(\ref{choiceofdelta})}{\leqslant} \left|e^{-\delta T} - 1\right|^2 M + \int_T^{\infty}\ e^{-2xt} |f(t)|^2\ dt \stackrel{(\ref{expoL2equivanishing})}{<} 2\eps,
\end{gather*} 

\noindent
which ends the proof. 
\end{proof}

A family $\Ffamily$ of Laplace-Pego functions with a common order $x\geqslant 0$ is said to be \textit{exponentially $L^2-$equicontinuous} at $x$, if
\begin{gather}
\forall_{\eps>0}\ \exists_{\delta > 0}\ \forall_{\substack{s\in (0,\delta)\\ f\in\Ffamily}}\ \left(\int_0^{\infty}\ e^{-2xt}|f(t) - f(t-s)|^2\ dt\right)^{\frac{1}{2}} < \eps.
\label{expL2equicontinuity}
\end{gather}

\noindent
Furthermore, we say that a family $\Ffamily$ is \textit{Laplace equivanishing} at $x$, if
\begin{gather}
\forall_{\eps>0}\ \exists_{T>0}\ \forall_{f\in\Ffamily}\ \int_{\reals\backslash[-T,T]}\ |\Laplace\{f\}(x+iy)|^2\ dy < \eps. 
\label{Laplaceequivanishing}
\end{gather}

We study the relationship between the novel notion of the exponential $L^2-$equicontinuity of $\Ffamily$ and the classical equicontinuity of $\Ffamily_x$ in the lemma below: 

\begin{lemma}
Let $\Ffamily$ be a family of Laplace-Pego functions with a common order $x\geqslant 0$. If $\Ffamily_x$ is $L^2-$bounded then $\Ffamily$ is exponentially $L^2-$equicontinuous at $x$ if and only if $\Ffamily_x$ is $L^2-$equicontinuous, i.e.
\begin{gather}
\forall_{\eps>0}\ \exists_{\delta > 0}\ \forall_{\substack{s\in(0,\delta),\\ f\in\Ffamily}}\ \int_0^{\infty}\ \left|e^{-x(t+s)}f(t+s) - e^{-xt}f(t)\right|^2\ dt < \eps.
\label{L2equicont}
\end{gather}
\label{expoequiandequicontinuity}
\end{lemma}
\begin{proof}
Since $\Ffamily_x$ is $L^2-$bounded, there exists $M>0$ such that 
$$\forall_{f\in\Ffamily}\ \left(\int_0^{\infty}\ e^{-2xt}|f(t)|^2\ dt\right)^{\frac{1}{2}} < M.$$

\noindent
We divide the proof of the lemma into two parts:

\vspace{0.3cm}
\noindent
\textbf{Part 1.}

In the first part of the proof, we assume that the family $\Ffamily$ is exponentially $L^2-$equicontinuous at $x$. We fix $\eps>0$ and choose $\delta > 0$ such that 
\begin{itemize}
	\item (\ref{L2equicont}) is satisfied, and
	\item for every $s \in (0,\delta)$ we have
		\begin{equation}
		\int_0^s\ e^{-2xt}|f(t)|^2\ dt < \eps,
		\label{choiceofdelta2}
		\end{equation}

		\noindent
		which is possible due to Theorem 8 in \cite{Natanson}, p. 148, and
	\item for every $s \in (0,\delta)$ we have
	\begin{gather}
	\left|e^{-xs} - 1\right|M < \eps.
	\label{choiceofdelta3}
	\end{gather}
\end{itemize}

\noindent
Consequently, for every $s\in(0,\delta)$ and $f\in\Ffamily$ we have
\begin{equation*}
\begin{split}
&\left(\int_0^{\infty}\ e^{-2xt}|f(t) - f(t-s)|^2\ dt\right)^{\frac{1}{2}} = \left(\int_{-s}^{\infty}\ e^{-2x(t+s)}|f(t+s) - f(t)|^2\ dt\right)^{\frac{1}{2}} \\
&\leqslant \left(\int_{-s}^{\infty}\ e^{-2x(t+s)}|f(t+s) - e^{xs}f(t)|^2\ dt\right)^{\frac{1}{2}} + \left(\int_{-s}^{\infty}\ e^{-2x(t+s)}|e^{xs}f(t) - f(t)|^2\ dt\right)^{\frac{1}{2}} \\
&\leqslant \bigg(\int_{-s}^0\ e^{-2x(t+s)}|f(t+s)|^2\ dt + \int_0^{\infty}\ e^{-2x(t+s)}|f(t+s) - e^{xs}f(t)|^2\ dt\bigg)^{\frac{1}{2}} \\
&+ \left(\int_0^{\infty}\ e^{-2x(t+s)}|f(t)|^2|e^{xs} - 1|^2\ dt\right)^{\frac{1}{2}}\\
&\leqslant \bigg(\int_0^s\ e^{-2xt}|f(t)|^2\ dt + \int_0^{\infty}\ |e^{-x(t+s)}f(t+s) - e^{-xt}f(t)|^2\ dt\bigg)^{\frac{1}{2}} \\
&+ |e^{-xs} - 1| M \stackrel{(\ref{L2equicont}), (\ref{choiceofdelta2}), (\ref{choiceofdelta3})}{\leqslant} (2\eps)^{\frac{1}{2}} + \eps.
\end{split}
\end{equation*}

\noindent
Since $\eps> 0$ was chosen arbitrarily, the above estimates end the first part of the proof.

\vspace{0.4cm}
\noindent
\textbf{Part 2.}

In this part of the proof, we assume that $\Ffamily_x$ is $L^2-$equicontinuous. Again, we fix $\eps> 0$ and let $\delta > 0$ be such that 
\begin{itemize}
	\item (\ref{expL2equicontinuity}) is satisfied, and
	\item for every $s\in(0,\delta)$ we have
	\begin{gather}
	\left|1-e^{xs}\right| M < \eps \hspace{0.4cm}\text{and}\hspace{0.4cm} e^{xs} \leqslant 2.
	\label{choosingdelta}
	\end{gather}
\end{itemize}

\noindent
For every $s\in(0,\delta)$ and $f\in\Ffamily$ we have
\begin{equation*}
\begin{split}
&\left(\int_0^{\infty}\ \left|e^{-x(t+s)}f(t+s) - e^{-xt}f(t)\right|^2\ dt\right)^{\frac{1}{2}} = \left(\int_s^{\infty}\ \left|e^{-xt}f(t) - e^{-x(t-s)}f(t-s)\right|^2\ dt\right)^{\frac{1}{2}} \\
&\leqslant \left(\int_s^{\infty}\ \left|e^{-xt}f(t) - e^{-x(t-s)}f(t)\right|^2\ dt\right)^{\frac{1}{2}} + \left(\int_s^{\infty}\ \left|e^{-x(t-s)}f(t) - e^{-x(t-s)}f(t-s)\right|^2\ dt\right)^{\frac{1}{2}} \\
&\leqslant \left|1 - e^{xs}\right| \left(\int_0^{\infty}\ e^{-2xt}|f(t)|^2\ dt\right)^{\frac{1}{2}} + e^{xs}\eps \stackrel{(\ref{choosingdelta})}{\leqslant} \left|1 - e^{xs}\right| M + 2\eps < 3\eps,
\end{split}
\end{equation*}

\noindent
which ends the proof.
\end{proof}

We will now study the relationship between the exponential $L^2-$equicontinuity and the Laplace equivanishing. First, let us recall the Minkowski inequality:

\begin{thm}(Minkowski integral inequality, comp. \cite{FollandRealAnalysis}, p. 194 or \mbox{\cite{SteinSingular}, p. 271})\\
Let $(X,\mu_X),\ (Y,\mu_Y)$ be $\sigma-$finite measure spaces, $1\leqslant p < \infty$ and let \mbox{$F:X\times Y \longrightarrow \complex$} be a measurable function. Then
\begin{gather}
\left(\int_X\ \left(\int_Y\ |F(x,y)|\ dy\right)^p\ dx\right)^{\frac{1}{p}} \leqslant \int_Y\ \left(\int_X\ |F(x,y)|^p\ dx\right)^{\frac{1}{p}}\ dy.
\label{Minkineq}
\end{gather}
\end{thm}

\begin{thm}
Let $\Ffamily$ be a Laplace-Pego family with a common order $x\geqslant 0.$ Exponential $L^2-$equicontinuity at $x$ implies Laplace equivanishing at $x$. Furthermore, if $\Ffamily_x$ is $L^2-$bounded, then the implication can be reversed. 
\label{expoequicontLaplaceequivanish}
\end{thm}
\begin{proof}
We divide the proof into two parts:

\vspace{0.4cm}
\noindent
\textbf{Part 1.} 

We assume that the family $\Ffamily$ is exponentially $L^2-$equicontinuous at $x$, so for a fixed $\eps > 0$ we can choose $\delta > 0$ according to the exponential $L^2-$equicontinuity (\ref{expL2equicontinuity}). Let $g$ be a compactly supported, continuous function on $\reals_+$, satisfying the following conditions:
\begin{itemize}
	\item $g$ is nonnegative, and
	\item $\supp(g)\subset (0,\delta),$ and
	\item $\int_0^{\infty}\ g(s)\ ds = 1.$
\end{itemize}

\noindent
Naturally, $g$ is a Laplace-Pego function of order $x.$

By Theorem \ref{RiemannLebesguelemma}, let $T>0$ be such that 
\begin{gather}
\forall_{y\in[-T,T]}\ |\Laplace\{g\}(x+iy)| \leqslant \frac{1}{2}.
\label{choiceofTRiemannLebesgue}
\end{gather}

\noindent
Consequently, we have
\begin{equation*}
\begin{split}
\forall_{f\in\Ffamily}\ \left(\int_{\reals\backslash[-T,T]}\ |\Laplace\{f\}(x+iy)|^2\ dy\right)^{\frac{1}{2}} &\leqslant \bigg(\int_{\reals\backslash[-T,T]}\ \bigg|\Laplace\{f\}(x+iy)\big(1 - \Laplace\{g\}(x+iy)\big)\bigg|^2\ dy\bigg)^{\frac{1}{2}}\\
&+ \left(\int_{\reals\backslash[-T,T]}\ |\Laplace\{f\}(x+iy)\Laplace\{g\}(x+iy)|^2\ dy\right)^{\frac{1}{2}}\\
\stackrel{(\ref{choiceofTRiemannLebesgue})}{\leqslant} \bigg(\int_{\reals\backslash[-T,T]}\ \bigg|\Laplace\{f\}(x+iy)\big(1 - \Laplace\{g\}(x&+iy)\big)\bigg|^2\ dy\bigg)^{\frac{1}{2}} + \frac{1}{2}\left(\int_{\reals\backslash[-T,T]}\ |\Laplace\{f\}(x+iy)|^2\ dy\right)^{\frac{1}{2}},
\end{split}
\end{equation*}

\noindent
which implies 
\begin{equation*}
\begin{split}
\forall_{f\in\Ffamily}\ \bigg(\int_{\reals\backslash[-T,T]}\ &|\Laplace\{f\}(x+iy)|^2\ dy\bigg)^{\frac{1}{2}} \leqslant 2\left(\int_{\reals\backslash[-T,T]}\ \bigg|\Laplace\{f\}(x+iy)\big(1 - \Laplace\{g\}(x+iy)\big)\bigg|^2\ dy\right)^{\frac{1}{2}} \\
&\stackrel{\text{Theorem } \ref{convolutiontheorem}}{\leqslant} 2 \left(\int_{\reals}\ |\Laplace\{f\}(x+iy) - \Laplace\{f\star g\}(x+iy)|^2\ dy\right)^{\frac{1}{2}} \\
&\stackrel{\text{Theorem } \ref{Planchereltheorem}}{=} 2\sqrt{2\pi} \left(\int_0^{\infty}\ e^{-2xt} |f(t) - f\star g(t)|^2\ dt\right)^{\frac{1}{2}} \\
&= 2\sqrt{2\pi} \left(\int_0^{\infty}\ e^{-2xt} \left|f(t) - \int_0^{\infty}\ f(t-s)g(s)\ ds\right|^2\ dt\right)^{\frac{1}{2}} \\
&= 2\sqrt{2\pi} \left(\int_0^{\infty}\ \left| \int_0^{\infty}\ e^{-xt} (f(t) - f(t-s))g(s)\ ds\right|^2\ dt\right)^{\frac{1}{2}} \\
&\stackrel{\text{Minkowski ineq.}}{\leqslant} 2\sqrt{2\pi} \int_0^{\infty}\ \left(\int_0^{\infty}\ e^{-2xt} |f(t) - f(t-s)|^2 |g(s)|^2\ dt\right)^{\frac{1}{2}}\ ds \\
&= 2\sqrt{2\pi} \int_0^{\infty}\ g(s) \left(\int_0^{\infty}\ e^{-2xt} |f(t) - f(t-s)|^2\ dt\right)^{\frac{1}{2}}\ ds \\
&= 2\sqrt{2\pi} \int_0^{\delta}\ g(s) \left(\int_0^{\infty}\ e^{-2xt} |f(t) - f(t-s)|^2\ dt\right)^{\frac{1}{2}}\ ds \\
&\leqslant 2\sqrt{2\pi} \eps \int_0^{\delta}\ g(s)\ ds = 2\sqrt{2\pi} \eps.
\end{split}
\end{equation*}

\noindent
Let us remark that the use of Minkowski inequality in hte above estimates is justified, because the function $F(t,s) = e^{-2xt} |f(t) - f(t-s)|^2 |g(s)|^2$ is measurable due to \mbox{Proposition 3.9}, p. 86 in \cite{SteinShakarchi}. Since $\eps >0$ was chosen arbitrarily, the above estimates end the first part of the proof. 

\vspace{0.3cm}
\noindent
\textbf{Part 2.}

For this part of the proof, we assume that $\Ffamily_x$ is $L^2-$bounded, so there exists $M_1>0$ such that 
\begin{gather}
\forall_{f\in\Ffamily}\ \int_0^{\infty}\ e^{-2xt}|f(t)|^2\ dt < M_1.
\label{M1exists}
\end{gather}

\noindent
We will show that if $\Ffamily$ is Laplace equivanishing at $x$ then it is exponentially $L^2-$equicontinuous at $x.$ For convenience, we denote $T_{-s}f(t) = f(t-s).$ We observe the following equalities 
\begin{equation}
\begin{split}
\forall_{\substack{s > 0\\ f\in\Ffamily}}\ \int_0^{\infty}\ e^{-2xt}|f(t)-f(t-s)|^2\ dt &= \int_0^{\infty}\ e^{-2xt}|f(t) - T_{-s}f(t)|^2\ dt \\
&\stackrel{\text{Theorem } \ref{Planchereltheorem}}{=} \frac{1}{2\pi}\ \int_{-\infty}^{\infty}\ |\Laplace\{f-T_{-s}f\}(x+iy)|^2\ dy \\
&= \frac{1}{2\pi}\ \int_{-\infty}^{\infty}\ |\Laplace\{f\}(x+iy) - e^{-s(x+iy)}\Laplace\{f\}(x+iy)|^2\ dy \\
&= \frac{1}{2\pi}\ \int_{-\infty}^{\infty}\ \left|1-e^{-s(x+iy)}\right|^2 |\Laplace\{f\}(x+iy)|^2\ dy.
\end{split}
\label{equalitywithTs}
\end{equation}

Fix $\eps>0$ and choose $T>0$ according to Laplace equivanishing (\ref{Laplaceequivanishing}). Let $\delta >0$ be such that 
\begin{gather}
\forall_{\substack{s\in(0,\delta)\\ y\in[-T,T]}}\ \left|1- e^{-s(x+iy)}\right|^2 < \eps,
\label{choiceofdelta4}
\end{gather}

\noindent
and put 
\begin{gather}
M_2 = \max_{\substack{s \in [0,\delta],\\ y\in\reals}}\ \left|1-e^{-s(x+iy)}\right|^2.
\label{choiceofM2}
\end{gather}

\noindent
Finally, for every $s\in(0,\delta)$ and $f\in\Ffamily$, we have
\begin{equation*}
\begin{split}
\int_0^{\infty}\ e^{-2xt}|f(t) - f(t-s)|^2\ dt \stackrel{(\ref{equalitywithTs})}{=}& \frac{1}{2\pi}\ \int_{-T}^T\ \left|1-e^{-s(x+iy)}\right|^2 |\Laplace\{f\}(x+iy)|^2\ dy \\
+& \frac{1}{2\pi}\ \int_{\reals\backslash [-T,T]}\ \left|1-e^{-s(x+iy)}\right|^2 |\Laplace\{f\}(x+iy)|^2\ dy \\
\stackrel{(\ref{choiceofdelta4}),\ (\ref{choiceofM2})}{\leqslant} \frac{\eps}{2\pi}\ \int_{-T}^T\ |\Laplace\{f\}(x+iy)|^2\ dy +& \frac{M_2}{2\pi}\ \int_{\reals\backslash [-T,T]}\ |\Laplace\{f\}(x+iy)|^2\ dy \\
\stackrel{\text{Theorem } \ref{Planchereltheorem},\ (\ref{Laplaceequivanishing})}{\leqslant} \eps\ \int_0^{\infty}\ e^{-2xt}|f(t)|^2\ dt +& \frac{M_2}{2\pi}\ \eps \leqslant \left(M_1 + \frac{M_2}{2\pi}\right) \eps.
\end{split}
\end{equation*}

\noindent
Since $\eps>0$ was chosen arbitrarily, we conclude the proof.
\end{proof}

Before we present the final theorem of the paper, let us recall the celebrated Riesz-Kolmogorov theorem:

\begin{thm}(Riesz-Kolmogorov theorem, comp. \cite{HancheOlsenHolden})\\
A family $\Afamily \subset L^2(\reals_+)$ is relatively compact if and only if 
\begin{itemize}
	\item $\Afamily$ is $L^2-$bounded, and
	\item $\Afamily$ is $L^2-$equicontinuous, and
	\item $\Afamily$ is $L^2-$equivanishing. 
\end{itemize}
\label{KolmogorovRiesz}
\end{thm}

The final theorem, which is the climax of the paper, should be juxtaposed with Pego theorem in \cite{Gorka}, \cite{GorkaKostrzewa} and \cite{Pego}. 

\begin{thm}
Let $\Ffamily$ be the family consisting of Laplace-Pego functions with a common order $x$ and such that $\Ffamily_x$ is $L^2-$bounded. The family $\Ffamily_x$ is relatively compact in $L^2(\reals_+)$ if and only if 
\begin{itemize}
	\item $\Ffamily$ is Laplace equicontinuous at $x$, and
	\item $\Ffamily$ is Laplace equivanishing at $x$.
\end{itemize}
\label{LaplacePego}
\end{thm}
\begin{proof} The proof is divided into two parts:

\vspace{0.3cm}
\noindent
\textbf{Part 1.}

We assume that $\Ffamily$ is Laplace equicontinuous and equivanishing at $x$. At first, we note that Laplace equicontinuity of $\Ffamily$ at $x$ implies that this family is exponentially $L^2-$equivanishing at $x$ (Theorem \ref{LaplaceequicontinuityexpoL2equvanishing}). In other words, $\Ffamily_x$ is $L^2-$equivanishing. 

Furthermore, Laplace equivanishing of $\Ffamily$ at $x$ implies that this family is exponentially $L^2-$equicontinuous (Theorem \ref{expoequicontLaplaceequivanish}). In other words, $\Ffamily_x$ is $L^2-$equicontinuous (Lemma \ref{expoequiandequicontinuity}). By Theorem \ref{KolmogorovRiesz}, we conclude that $\Ffamily_x$ is relatively compact in $L^2(\reals_+)$.

\vspace{0.3cm}
\noindent
\textbf{Part 2.}

For the second part of the proof, we assume that $\Ffamily_x$ is relatively compact in $L^2(\reals_+)$. By Theorem \ref{KolmogorovRiesz}, the family $\Ffamily_x$ is $L^2-$equicontinuous and $L^2-$equivanishing. 

$L^2-$equicontinuity of $\Ffamily_x$ implies that $\Ffamily$ is Laplace equivanishing at $x$ (Lemma \ref{expoequiandequicontinuity} and Theorem \ref{expoequicontLaplaceequivanish}). Moreover, $L^2-$equivanishing of $\Ffamily_x$ implies that $\Ffamily$ is Laplace equicontinuous at $x$ (Theorem \ref{LaplaceequicontinuityexpoL2equvanishing}), which ends the proof. 
\end{proof}

\end{document}